\documentclass[10pt]{article}
\font\smallit=cmti10

\usepackage{amssymb,latexsym,amsmath,epsfig,amsthm} 

\makeatletter

\renewcommand\section{\@startsection {section}{1}{\z@}
{-30pt \@plus -1ex \@minus -.2ex}
{2.3ex \@plus.2ex}
{\normalfont\normalsize\bfseries}}

\renewcommand\subsection{\@startsection{subsection}{2}{\z@}
{-3.25ex\@plus -1ex \@minus -.2ex}
{1.5ex \@plus .2ex}
{\normalfont\normalsize\bfseries}}

\renewcommand{\@seccntformat}[1]{\csname the#1\endcsname. }

\makeatother
\newcommand{\euler}[2]{\genfrac{\langle}{\rangle}{0pt}{}{#1}{#2}}
\newtheorem{theorem}{Theorem}
\newtheorem{lemma}{Lemma}

\begin{document}

\begin{center}
\uppercase{\bf Rectified Simplex Polytope Numbers}
\vskip 20pt
{\bf Michael A. Jackson}\\
{\smallit Department of Mathematics, Grove City College, Grove City, Pennsylvania 16127, USA}\\
{\tt majackson@gcc.edu}\\ 
\vskip 10pt
{\bf Doug Smith}\\
{\smallit Department of Mathematics, Grove City College, Grove City, Pennsylvania 16127, USA}\\
\vskip 10pt
{\bf Peter Jantsch}\\
{\smallit Department of Mathematics, Grove City College, Grove City, Pennsylvania 16127, USA}\\
\vskip 10pt
{\bf Christina Scurlock}\\
{\smallit Department of Mathematics, Grove City College, Grove City, Pennsylvania 16127, USA}\\
\vskip 10pt
{\bf Chelsea Snyder}\\
{\smallit Department of Mathematics, Grove City College, Grove City, Pennsylvania 16127, USA}\\
\vskip 10pt
{\bf Eric Fairchild}\\
{\smallit Department of Mathematics, Grove City College, Grove City, Pennsylvania 16127, USA}\\
\vskip 10pt
{\bf Robin Mabe}\\
{\smallit Department of Mathematics, Grove City College, Grove City, Pennsylvania 16127, USA}\\
\vskip 10pt
{\bf Emma Polaski}\\
{\smallit Department of Mathematics, Grove City College, Grove City, Pennsylvania 16127, USA}\\
\end{center}
\vskip 30pt
\centerline{\smallit Received: , Revised: , Accepted: , Published: } 
\vskip 30pt

\centerline{\bf Abstract}
\noindent
Polytope numbers for a given polytope are an integer sequence defined by the combinatorics of the polytope. Recent work by H. K. Kim and J. Y. Lee has focused on writing polytope number sequences as sums of simplex number sequences. We focus on $r$-rectified simplices and show that the sequences for these polytopes can be written as alternating sums of simplex numbers analogous to the inclusion-exclusion given by the geometric process of $r$-rectification.

\section{Introduction}
Polygonal numbers are sequences of integers based on the set of points forming polygons on the plane. These number sequences date back to the ancient Greeks. The polygonal numbers have been generalized to polyhedral numbers by the ancient Greeks and more recently to all higher dimensions by Kim \cite{hkk:pn} and then by Kim and Lee \cite{hkk:pnp}. For more on polytope numbers see also \cite{dd:fn}.

For a $d$-dimensional polytope $P$, we denote the set of faces of the polytope $P$ by $\mathcal{F}(P)$, and for a dimension $0\leq i\leq d$, we denote the $i$-dimensional faces of $P$ by $\mathcal{F}_i (P)$. In this definition of polytope numbers, we will assume that our polytope $P$ is uniform (vertex transitive) and that each face of $P$ is also a uniform polytope. We will assume that we have chosen a particular vertex of $P$, which we will denote $\mathbf{v}_P$.

Now we can define the polytope number sequence $\{P(n)\}_{n\geq 0}$ for the polytope $P$ as well as the interior number sequence $\{P(n)^\# \}_{n\geq 0}$. When $d=0$, $P(0)=P(0)^\#=0$, and $P(n)=P(n)^\#=1$ for $n\geq 1$. Now we assume that when $d>0$, and that for each $F\in \mathcal{F}_k (P)$ with $0\leq k<d$, the sequences $F(n)$ and $F(n)^\#$ are defined. We again start the sequences with $P(0)=P(0)^\#=0$,  $P(1)=1$, and $P(1)^\#=0$. The rest of the sequence is defined recursively by 
$$P(n)=P(n-1)+\sum _{\mathbf{v}_P\notin F \in \mathcal{F}(P)} F(n)^\# \textrm{ and }P(n)^\#=P(n)-\sum _{F\in \mathcal{F}(P)\setminus \{P\} } F(n)^\# .$$

Another way of thinking about the polytope number sequence is as a method of counting the number of points used to construct nested versions of the polytope each sharing one common vertex. This construction is done recursively where each edge coming from the given vertex $\mathbf{v}_P$ is extended by one additional point when using the $n-1$-indexed shape to create the $n$-indexed shape. Then each face of $P$ that does not contain $\mathbf{v}_P$ is added to the shape to complete an exterior where each edge has $n$ points and the faces are all complete $n$-indexed versions from the sequence of that face's polytope sequence. For more on this process see \cite[Section 1]{hkk:pn} and \cite[Section 3]{hkk:pnp}.

H. K. Kim gives formulas for the $d$-dimensional regular polytope number sequences \cite{hkk:pn}. We use $\alpha^d(n)$, $\beta^d (n)$, and $\gamma^d(n)$ to denote the $n^{th}$ number in the polytope sequences for the $d$ dimensional simplex, cross-polytope, and measure-polytope respectively. With this notation, Kim's formulas are
$$\alpha^d(n)=\binom{n+d-1}{d} \textrm{ , }\hspace{.2 in} \beta ^d(n)=\sum _{i=0}^{d-1} \binom{d-1}{i} \alpha^d(n-i)$$
$$\textrm{ and } \gamma^d(n)=\sum _{i=0}^{d-1} \euler{d}{i} \alpha^d(n-i).$$
It has been shown that any convex polytope number sequence for a $d$-dimensional polytope can be written as $$p^d(n)=\sum _{i=0}^{d-1} a_i \alpha^d(n-i)$$
where $a_0=1$ and $a_i\geq 0$ for each $i$. (See \cite{hkk:pnp} and \cite{j:oipns}.) This formula will be called the $d$-dimensional simplex decomposition of the polytope number sequence. 

Given a polytope $\mathcal{P}^d$ of dimension $d$, if $1\leq r<d$, then the $r$-rectification of $\mathcal{P}^d$ is the result of cutting off each vertex of $\mathcal{P}^d$ to the center of each $r$-face that contains that vertex. In the case where $\mathcal{P}^d$ is a regular polytope, the resulting $r$-rectified polytope will be uniform. Notice that the $(d-1)$-rectification of $\mathcal{P}^d$ will result in the dual polytope. 

For the polytope $\mathcal{P}^d$, we will use $\mathcal{F}_k(\mathcal{P}^d)$ to denote the set of $k$-faces of $\mathcal{P}^d$. Notice that if $r>1$, then the center portion of each edge will be cut off twice. So using the idea of inclusion-exclusion, we must add the center portion of each edged back in. In a similar fashion, if $r>2$, then the center of each 2-face is cut off the number of times equal to the number of vertices of that face. However, when the the center of the edges are added back in the center of the 2-face will be added in a number of times equal to the edges of that 2-face. Since each 2-face has an equal number of vertices and edges we are left to subtract one additional center of the 2-face. This idea continues for each dimension less than $r$. 

The inclusion-exclusion principle is mirrored in the formula for the $r$-rectified simplex polytope number sequences. We will use $\lambda _r ^d(n)$ to represent the $n^{th}$ term of the polytope number sequence for the $r$-rectified $d$-dimensional simplex where $r<d$. Recall that if $V^d(n)$ is the polytope number sequence for a given polytope, we use $V^d(n)^\#$ to represent the interior number sequence. This sequence is given by taking $V^d(n)$ and subtracting the exterior points for this polytope in the $d$-dimensional space it inhabits.

\begin{theorem}\label{thm:hrect}
For the case where $0\leq r< d$ and $n\geq 1$, we have the following formula for the polytope number sequence and the corresponding interior number sequence: 
$$\lambda^{d}_{r}(n)=\sum\limits^{r}_{i=0}(-1)^{r-i}\binom{d+1}{r-i}\alpha^d((i+1)n-r)$$
$$\lambda^{d}_{r}(n)^\#=\sum\limits^{r}_{i=0}(-1)^{r-i}\binom{d+1}{r-i}\alpha^d((i+1)n+r-2i)^\#.$$
\end{theorem}

It is of interest to note that the rectified versions of the other regular polytopes do not have number sequences that satisfy a similar formula based on the inclusion-exclusion arising from the geometric rectification.

We are interested in this formula to find the coefficients of the $d$-dimensional simplex decomposition of the $r$-rectified $d$-dimensional simplex numbers. To find these coefficients we need to relate the number sequences $\alpha^d((i+1)n-r)$ to $\alpha^d(n-j)$ where $j\in [d-1]_0$.

\begin{theorem}\label{th:bigalpha}
Let $n,a\geq 1$ and $b\geq 0$, then 
$$\alpha^d(an-(a-1)-b)=\sum_{j\geq 0} \sum_{i=0}^{j} (-1)^i \binom{d+1}{i} \binom{d+a(j-i)-b}{a(j-i)-b} \alpha^d(n-j).$$
\end{theorem}
Notice that the formula can also be written as follows:
$$ \alpha^d(an-(a-1)-b)=\sum_{j\geq 0} c_j \alpha^d(n-j)$$
$$\textrm{where } \; \sum_{j\geq 0} c_j x^j = (1-x)^{d+1} \sum _{k\geq 0} \binom{d+ak-b}{ak-b} x^k .$$
If we define the generalized binomial coefficients of order $s$, $\binom{n}{m}_s$ as in \cite{bb:gpsp}, then the formula from Theorem \ref{thm:hrect} becomes:
$$\lambda^d_r(n)=\sum _{j=0}^{d-1} \left[ \sum_{i=0}^r (-1)^{r-i} \binom{d+1}{r-i} \binom{d+1}{(i+1)j+i-r}_{i+1} \right] \alpha^d(n-j).$$

The rest of the paper will lead to proofs of Theorems \ref{thm:hrect} and \ref{th:bigalpha}. In Section \ref{sec:id}, we show two important identities for the proofs, while Section \ref{sec:lem}, we prove necessary lemmas for the proof of Theorem \ref{thm:hrect}. The proof of Theorem \ref{thm:hrect} is in Section \ref{sec:th1}, and the proof of Theorem \ref{th:bigalpha} is in Section \ref{sec:thm2}.

We would like to acknowledge the Swezey Fund of Grove City College for funding this research.

\section{Identities}
\label{sec:id}
Before we give proofs for the two theorems in the introduction, we need to show identities that we will use extensively in the proofs.

\begin{lemma}[Alternating Vandermonde Indentity]\label{lem:altvan}
If $b,c,n>0$, then 
$$\sum_{k=0}^n (-1)^k \binom{c-k}{b-k} \binom{n}{k} = \binom{c-n}{b}.$$
\end{lemma}
We refer to this as the alternating Vandermonde identity for its similarity to Vandermonde's convolution.
\begin{proof}
Recall Vandermonde's convolution:
$$\sum_{k=0}^n \binom{r}{k}\binom{s}{n-k}=\binom{r+s}{n}.$$
Now we will start the alternating version. We will make use of negating the upper index, which states $\binom{r}{k}=(-1)^k \binom{k-r-1}{k}$.
\begin{align*}
\sum_{k=0}^n (-1)^k \binom{c-k}{b-k} \binom{n}{k} &= \sum_{k=0}^n (-1)^k (-1)^{b-k} \binom{b-c-1}{b-k}\binom{n}{k}\\
&=(-1)^b \sum_{k=0}^n \binom{b-c-1}{b-k}\binom{n}{k}\\
&=(-1)^b \binom{b-c-1+n}{b}\\
&= \binom{c-n}{b}.
\end{align*}
\end{proof}
In \cite[Section 3.2]{hkk:pnp} we are told that if one facet of the array defining a simplex number $\alpha^d(n)$ is cut away, the resulting set of points will be of size $\alpha^d(n-1)$. This follows from the identity $\alpha^d(n)-\alpha^{d-1}(n)=\alpha^d(n-1)$, which, when interpreted in terms of binomial coefficients, is just Pascal's identity. If we use $k$ facet cuts on $\alpha^d(n)$, this yields $\alpha^d(n-k)$. Since the exterior of a $d$-dimensional simplex is the $d+1$ facets, $\alpha^d(n)^\#=\alpha^d(n-d-1)$ (this is equivalent to doing $d+1$ facet cuts on $\alpha^d(n)$). Thus in terms of binioial coefficients, $\alpha^d (n) ^\sharp  = \binom{n-2}{d}$.
\begin{lemma}[Generalized Kim's Identity]\label{lem:kimid}
For $d>k\geq 0$ and $j\geq 0$,
$$\sum_{m=0}^{d-1} \binom{d-1-k}{m-k} \alpha^{m-k-1+j}(n)^\# = \alpha^{d-k+j-2}(n-j).$$
\end{lemma}
Notice that the case for $k=j=0$ is due to H.K. Kim \cite[p. 6]{hkk:pn}.
\begin{proof}
First we recall the Vandermonde convolution:
$$\sum _{j=0} ^{n} \binom{n}{j}\binom{m}{k-j} = \binom{n+m}{k}.$$
This can be rewritten:
$$\sum _{j=0} ^{n} \binom{n}{j}\binom{m}{m-k+j} = \binom{n+m}{k}$$
or by setting $m-k=p$
$$\sum _{j=0} ^{n} \binom{n}{j}\binom{m}{j+p} = \binom{n+m}{m-p} =\binom{n+m}{n+p}.$$
Now for our identity:
\begin{align*} \sum_{m=0}^{d-1} \binom{d-1-k}{m-k} \alpha^{m-k-1+j}(n)^\# &= \sum_{m=0} ^{d-1} \binom{d-1-k}{m-k} \binom{n-2}{m-k-1+j} \\
&= \binom{d+n-k-3}{d-k+j-2}\\
&= \alpha^{d-k+j-2}(n-j).
\end{align*}
\end{proof}
\section{Lemmas}
\label{sec:lem}
Before we prove the theorem we will prove several lemmas that will be useful in our proof.
\begin{lemma}\label{lem:claim1}
For $r,d\geq 0$ and $n\geq 1$,
$$\sum\limits^{r}_{k=0}\sum\limits^{d-r+k}_{m=1} \binom{d+1}{ r-k}\binom{d+1-r+k}{m+1} \sum\limits^{k}_{i=0}(-1)^{k-i}\binom{m+1}{k-i}\alpha^m((i+1)n+k-2i)^\#$$
$$\hspace{1 in} = \sum_{i=0}^r (-1)^{r-i} \binom{d+1}{r-i}\alpha^d ((i+1)n-r).$$
\end{lemma}
\begin{proof}
Start with the left hand side of the equation and then re-index by replacing $m$ by $m-r+k+1$.
\begin{align*}
&\sum\limits^{r}_{k=0}\sum\limits^{d-r+k}_{m=1} \binom{d+1}{ r-k}\binom{d+1-r+k}{m+1}\sum\limits^{k}_{i=0}(-1)^{k-i}\binom{m+1}{k-i}\alpha^m((i+1)n+k-2i)^\#\\
&=\sum\limits^{r}_{k=0}\sum\limits^{d-1}_{m=r-k}\sum\limits^{k}_{i=0} (-1)^{k-i}\binom{d+1}{r-k}\binom{m-r+k+2}{k-i}\binom{d+1-r+k}{m-r+k+2}\alpha^{m-r+k+1}((1+i)n+k-2i)^\#\\
&=\sum\limits^{r}_{k=0}\sum\limits^{d-1}_{m=r-k}\sum\limits^{k}_{i=0} (-1)^{k-i}\binom{d+1}{r-k}\binom{m-r+k+2}{k-i}\binom{d+1-r+k}{m-r+k+2}\alpha^{m-r+k+1}((1+i)n+k-2i)^\#\\
&=\sum\limits^{r}_{k=0}\sum\limits^{k}_{i=0} (-1)^{k-i}\binom{d+1}{r-k}\binom{d+1-r+k}{k-i}\sum\limits^{d-1}_{m=r-k}\binom{d+1-r+i}{m-r+i+2}\alpha^{m-r+k+1}((1+i)n+k-2i)^\#.
\end{align*}
Use the generalized Kim's identity on the $m$ summation and reverse the subset of a subset identity for the other binomial coefficients to result in the following:
\begin{align*}
&=\sum\limits^{r}_{k=0}\sum\limits^{k}_{i=0} (-1)^{k-i}\binom{d+1}{r-i}\binom{r-i}{r-k} \alpha^{d-r+k}((1+i)n-i)\\
&=\sum\limits^{k}_{i=0}\sum\limits^{r}_{k=0} (-1)^{k-i}\binom{d+1}{r-i}\binom{r-i}{r-k}\binom{(1+i)n-i+d+r+k-1}{d+r+k}\\
&\textrm{(re-index by replacing $k$ by $r-k$)}\\
&=\sum\limits^{k}_{i=0}\binom{d+1}{r-i}\sum\limits^{r}_{k=0}\binom{r-i}{k} (-1)^{(r-i)-k}\binom{(1+i)n-i+d-k-1}{d-k}\\
&\textrm{(use the alternating Vandermonde identity)}\\
&=\sum\limits^{r}_{i=0}\binom{d+1}{r-i} (-1)^{r-i}\binom{(1+i)n-i+d-1-r+i}{d}\\
&=\sum\limits^{r}_{i=0}(-1)^{r-i}\binom{d+1}{r-i}\alpha^d((i+1)n-r).
\end{align*}
\end{proof}
\begin{lemma}\label{lem:claim2}
For $d,r\geq 0$ and $n\geq 2$,
$$\sum\limits^{r}_{k=0}\sum\limits^{d-r+k}_{m=r-1} \binom{r+1}{r-k}\binom{d-r}{m-k} \sum\limits^{k}_{i=0}(-1)^{k-i}\binom{m+1}{k-i}\alpha^m((i+1)n+k-2i)^\#$$
$$\hspace{1 in}=\sum_{i=0}^r (-1)^{r-i}\binom{d+1}{r-i} \alpha^d((i+1)(n-1)-r).$$
\end{lemma}
\begin{proof}
Start with the left hand side of the equation and then re-index by replacing $m$ by $m-r+k+1$.
\begin{align*}
&\sum\limits^{r}_{k=0}\sum\limits^{d-r+k}_{m=r-1} \binom{r+1}{r-k}\binom{d-r}{m-k} \sum\limits^{k}_{i=0}(-1)^{k-i}\binom{m+1}{k-i}\alpha^m((i+1)n+k-2i)^\#\\
&=\sum\limits^{r}_{k=0}\sum\limits^{d-1}_{m=r-1}\sum\limits^{k}_{i=0} (-1)^{k-i}\binom{r+1}{r-k}\binom{d-r}{m-r+1}\binom{m-r+k+2}{k-i}\alpha^{m-r+k+1}((1+i)n+k-2i)^\#\\
&=\sum\limits^{r}_{k=0}\sum\limits^{d-1}_{m=r-1}\sum\limits^{k}_{i=0}\sum ^{r-k}_{j=0} \left[(-1)^{k-i}\binom{r+1}{r-k}\binom{d-r}{m-r+1}\binom{m-r+k+2}{k-i}\right.\\
&\hspace{1 in}\left. \cdot (-1)^j\binom{r-k}{j}\alpha^{m+1}((1+i)n+r-2i-j)^\# \right]\\
&=\sum\limits^{d-1}_{m=r-1}\sum\limits^{r}_{k=0}\sum\limits^{k}_{i=0} \sum\limits^{r-k}_{j=0} \left[(-1)^{k-i+j}\binom{r+1}{r-k}\binom{d-r}{m-r+1}\binom{m-r+k+2}{k-i}\right.\\
&\hspace{1 in}\left. \cdot \binom{r-k}{j}\binom{(1+i)n+r-2i-2-j}{m+1}\right]\\
&\textrm{(use the subset of a subset identity and rearrange the order of the sums)}\\
&=\sum\limits^{d-1}_{m=r-1}\sum\limits^{r}_{i=0}\sum\limits^{r-i}_{j=0}\left(\left[ \sum\limits^{r-j}_{k=i}  (-1)^k \binom{r+1-j}{r-k-j}\binom{m-r+k+2}{k-i}\right]\right.\\
&\hspace{1 in}\left. \cdot (-1)^{j-i}\binom{r+1}{j}\binom{d-r}{m-r+1}\binom{(1+i)n+r-2i-2-j}{m+1}\right).
\end{align*}
Now we will focus just on the inner $k$ summation. First we will replace $k$ by $r-k$.
\begin{align*}
\sum\limits^{r}_{k=0}(-1)^k \binom{r+1-j}{r-k-j}\binom{m-r+k+2}{k-i} &=\sum\limits^r_{k=0}(-1)^{r-k}\binom{r+1-j}{k-j}\binom{m-k+2}{r-k-i}\\
\textrm{(replace $k-j$ by $p$)}\hspace{.2 in} &=\sum\limits^{r-j-i}_{p=0}(-1)^{r-j}(-1)^p \binom{r-j+1}{p}\binom{m+2-j-p}{r-i-j-p}\\
\textrm{(apply the alternating Vandermonde}&\textrm{ identity)}\\
&=(-1)^{r-j}\binom{m+2-r-1}{r-i-j}=(-1)^{r-j}\binom{m-r+1}{r-i-j}.
\end{align*}
Now put this simplification back with the other summations.
\begin{align*}
&\sum\limits^{d-1}_{m=r-1}\sum\limits^{k}_{i=0} \sum\limits^{r-k}_{j=0} (-1)^{-i+j}(-1)^{r-j} \binom{m-r+1}{r-i-j}\binom{r+1}{j}\binom{d-r}{m-r+1}\binom{((1+i)n+r-2i-2-j}{m+1}\\
&=\sum\limits^{k}_{i=0} \sum\limits^{r-k}_{j=0} (-1)^{r-i}\binom{r+1}{j}\binom{d-r}{r-i-j}\sum\limits^{d-1}_{m=r-1}\binom{d-2r+i+j}{m-2r+i+j+1}\alpha^{m+1}((1+i)n+r-2i-j)^\#\\
&\textrm{(apply the generalized Kim's identity)}\\
&=\sum\limits^{k}_{i=0} \sum\limits^{r-k}_{j=0} (-1)^{r-i}\binom{r+1}{j}\binom{d-r}{r-i-j}\alpha^{d}((1+i)n-r-i-1)\\
&\textrm{(apply Vandermonde's convolution on the $j$ summation)}\\
&=\sum\limits^{k}_{i=0}(-1)^{r-i}\binom{d+1}{r-i}\alpha^{d}((1+i)n-r-i-1)\\
&=\sum\limits^{k}_{i=0}(-1)^{r-i}\binom{d+1}{r-i}\alpha^{d}((1+i)(n-1)-r).
\end{align*}
\end{proof}
\begin{lemma}\label{lem:claim3}
For $d,r\geq 0$,
$$\sum_{k=0}^r (-1)^k \binom{d+1}{r-k}\binom{d+1-r+k}{k+1}=\binom{d+1}{r+1}.$$
\end{lemma}
\begin{proof}
We start by doing a reverse subset of a subset identity and then a column sum identity.
\begin{align*}
\sum_{k=0}^r (-1)^k \binom{d+1}{r-k}\binom{d+1-r+k}{k+1}&= \sum_{k=0}^r (-1)^k \binom{d+1}{k+1}\binom{d-k}{r-k}\\
&=\sum_{k=0}^r (-1)^k \sum_{i=0}^d \binom{i}{k}\binom{d-k}{r-k}\\
\textrm{(apply the alternating Vandermonde identity)}&\\
&=\sum_{i=0}^d \binom{d-i}{r}=\sum_{i=0}^d \binom{i}{r}=\binom{d+1}{r}.
\end{align*}
In the final line we replace $i$ by $d-i$ and then use the column sum identity.
\end{proof}

\begin{lemma}\label{lem:claim4}
For $r\geq 0$,
$$\sum_{k=0}^r (-1)^k \binom{r+1}{r-k} =1.$$
\end{lemma}
\begin{proof}
Proceed by replacing $k$ by $r-k$ and then using the alternating sum of a row of Pascal's triangle.
$$\sum_{k=0}^r (-1)^k \binom{r+1}{r-k} =\sum_{k=0}^r (-1)^{r-k} \binom{r+1}{k}= (-1)^r (-1)^r=1.$$
\end{proof}
\section{Proof of Theorem \ref{thm:hrect}}
\label{sec:th1}
\begin{proof}
This proof will use a triple induction. We will induct on $r$, then on $d$, and finally on $n$. In working with the $d$ induction, we will show that the standard $r$-rectification sequence holds for dimension $d$ and all $n\geq 0$ using induction on $n$; then we will show that the interior $r$-rectification sequence holds for all $n\geq 0$ using the standard polytope sequence.

Also we will use the notation $\lambda _r^d(n)$ and $\lambda _r^d(n)^\#$ even in the case where $0<d\leq r$. These are no longer true polytope number sequences but will be defined by the formulas in the statement of Theorem \ref{thm:hrect}. So for $0<d\leq r$, we define
$$\lambda^{d}_{r}(n)=\sum\limits^{r}_{i=0}(-1)^{r-i}\binom{d+1}{r-i}\alpha^d((i+1)n-r)$$
$$\lambda^{d}_{r}(n)^\#=\sum\limits^{r}_{i=0}(-1)^{r-i}\binom{d+1}{r-i}\alpha^d((i+1)n+r-2i)^\#.$$
We start with the induction on $r$ and a base case of $r=0$. Notice that a $0$-rectified $d$-simplex is just a $d$-simplex. The formulas hold since they reduce to
$$\lambda^d_0 (n)= \alpha^d(n) \textrm{ and } \lambda^d_0 (n)^\#= \alpha^d(n)^\#.$$
Now we will assume that for each $k<r$, the formulas hold for $\lambda ^m_k  (n)$ and $\lambda ^m_k (n)^\#$ where $m> k$ and $n\geq 1$.

We will now fix the $r\geq 1$ and induct on $d$ starting with $d=1$. We are proving the equations in Theorem \ref{thm:hrect} for $d>r$; however, for $d<r$ we will show that $\lambda_r^d(n)^\#=0$ for $n>0$, and for $d=r$ we will show that $\lambda_r^r(n)^\#=(-1)^r$ for $n\geq 1$.

Next we will induct on $d$ having fixed $r$. We will assume the formula holds for $\lambda_k^m(n)^\#$ for all $k<r$ with any $m>0$ and for $m<d$ when $k=r$.

We will look at three cases: when $0<d<r$, when $0<d=r$, and when $0<r<d$. For the case where $0<d<r$, we will first show that $\lambda _r^d (n)^\# = 0 $ for $n>1$.\\
Using Lemma \ref{lem:claim2}, since $r>d$, $\lambda _r^d (n)=0$. Now using Lemma \ref{lem:claim1},
\begin{align*}
0&=\sum\limits^{r}_{k=0}\sum\limits^{d-r+k}_{m=1} \binom{d+1}{r-k}\binom{d+1-r+k}{m+1} \sum\limits^{k}_{i=0}(-1)^{k-i}\binom{m+1}{k-i}\alpha^m((i+1)n+k-2i)^\#\\
&=\sum\limits^{r}_{k=0}\sum\limits^{d-r+k}_{m=1} \binom{d+1}{r-k}\binom{d+1-r+k}{m+1} \lambda_k^m(n)^\#\\
&=\binom{d+1}{0}\binom{d+1}{d+1} \lambda_r^d(n)^\#.
\end{align*}
This last statement holds since in each term of the summation $m<k$; we can apply induction for each case where $m<d$ or $k<r$.

Next in our induction is the case where $d=r$; to show that $\lambda _r^r(n)=1$ for $n\geq 1$, we will need to induct on $n$. We start with $n=1$: 
$$\lambda^r_r(n)=\sum\limits^{r}_{i=0}(-1)^{r-i}\binom{r+1}{r-i}\alpha^r((i+1)-r).$$
The simplex number here is non-zero only when $i=r$. So we have 
$$\lambda^r_r(n)=\sum\limits^{r}_{i=0}(-1)^{r-i}\binom{r+1}{r-i}\alpha^r((i+1)-r)=\binom{r+1}{0}\alpha^r(1)=1.$$
Using Lemmas \ref{lem:claim1} and \ref{lem:claim2}, we can see that for $n>1$,
\begin{align*}
&\lambda^{r}_{r}(n)-\lambda^{r}_{r}(n-1)\\
&=\sum\limits^{r}_{i=0}(-1)^{r-i}\binom{r+1}{r-i}\alpha^r((i+1)n-r)-\sum\limits^{r}_{i=0}(-1)^{r-i}\binom{r+1}{r-i}\alpha^r((i+1)(n-1)-r)\\
&=\sum\limits^{r}_{k=0}\sum\limits^{k}_{m=1} \left[\binom{r+1}{r-k}\binom{k+1}{m+1}-\binom{r+1}{r-k}\binom{0}{m-k}\right] \lambda^{m}_{k}(n)^\#\\
\textrm{(this is only nonzero}&\textrm{ when $m=k$)}\\
&=\sum\limits^{r}_{k=0} \left[\binom{r+1}{r-k}\binom{k+1}{k+1}-\binom{r+1}{r-k}\binom{0}{0}\right] \lambda^{k}_{k}(n)^\#\\
&=0.
\end{align*}
The last equality holds since each coefficient is zero. This implies that $\lambda _r^r(n)=\lambda _r^r(1)=1$ for all $n\geq 1$. Now looking at the interior number, we use Lemma \ref{lem:claim1} and recall that $\lambda_k^m(n)^\#=0$ for $0<m<k\leq r$ and $\lambda_k^k(n)^\#=(-1)^k$ for $0<k<r$.
\begin{align*}
\lambda _r^r(n)=1&=\sum\limits^{r}_{k=0}\sum\limits^{k}_{m=1} \binom{r+1}{r-k}\binom{k+1}{m+1} \sum\limits^{k}_{i=0}(-1)^{k-i}\binom{m+1}{k-i}\alpha^m((i+1)n+k-2i)^\#\\
&=\sum\limits^{r}_{k=0}\sum\limits^{k}_{m=1} \binom{r+1}{r-k}\binom{k+1}{m+1} \lambda_k^m(n)^\#\\
&=\lambda_r^r(n)^\#+\sum\limits^{r-1}_{k=0} \binom{r+1}{r-k}\binom{k+1}{k+1}  \lambda_k^k(n)^\# \\
&=\lambda_r^r(n)^\#+\sum\limits^{r-1}_{k=0} \binom{r+1}{k+1}(-1)^k \\
&=\lambda_r^r(n)^\#+\sum\limits^{r}_{k=1} \binom{r+1}{k}(-1)^{k+1}.
\end{align*}
In the last step we reindexed. Now we solve for $\lambda_r^r(n)$; notice we get a partial alternating sum of a row of Pascal's triangle:
$$\lambda_r^r(n)^\#=1+\sum\limits^{r}_{k=1} \binom{r+1}{k}(-1)^{k}=\sum\limits^{r}_{k=0} \binom{r+1}{k}(-1)^{k}=(-1)^r.$$

Now we will assume that for every dimension less than $d$, both formulas hold and we will show the polytope number sequence formula for dimension $d$.
Notice that for true polytope number sequences, the $n=1$ term is always 1. Thus $\lambda_r^d(1)=1$ for all $r<d$. Also notice that if $r<d$, the formula on the right hand side of the first equation in Theorem \ref{thm:hrect} becomes
$$\sum\limits^{r}_{i=0}(-1)^{r-i}\binom{d+1}{r-i}\alpha^d((i+1)-r)$$
where only the $i=r$ term is non-zero. This sum becomes $\binom{d+1}{0}\alpha^d(1)=1$.\\
By induction on $n$, it is enough to show the following formula for $n>1$:
\begin{align*}
& \lambda^{d}_{r}(n)-\lambda^{d}_{r}(n-1)\\
&=\sum\limits^{r}_{i=0}(-1)^{r-i}\binom{d+1}{r-i}\alpha^d((i+1)n-r)-\sum\limits^{r}_{i=0}(-1)^{r-i}\binom{d+1}{r-i}\alpha^d((i+1)(n-1)-r).
\end{align*}
Using Kim's process we have the following formulation:
\begin{align*}
&\lambda_r^d(n)-\lambda_r^d(n-1)\\
&=\sum\limits^{r}_{k=0}\sum\limits^{d-r+k}_{m=k+1} \left[\binom{d+1}{r-k}\binom{d+1-r+k}{m+1}-\binom{r+1}{r-k}\binom{d-r}{m-k}\right] \lambda^{m}_{k}(n)^\#+\binom{d+1}{r+1}-1.
\end{align*}
Notice that the term when $k=r$ and $d=m$ is zero because the coefficient is zero.\\
By Lemmas \ref{lem:claim3} and \ref{lem:claim4},
\begin{align*}
\binom{d+1}{r+1}-1&=\sum_{k=0}^r (-1)^k \left[\binom{d+1}{r-k}\binom{d+1-r+k}{k+1}-\binom{r+1}{r-k} \right]\\
&=\sum_{k=0}^r \left[\binom{d+1}{r-k}\binom{d+1-r+k}{k+1}-\binom{r+1}{r-k} \right]\lambda _k^k(n)^\#.
\end{align*}
$\lambda _k^m(n)^\# = 0$ for all $0<m<k$; thus we can add these terms into the sequence as well, giving us
$$\lambda_r^d(n)-\lambda_r^d(n-1)=\sum\limits^{r}_{k=0}\sum\limits^{d-r+k}_{m=1} \left[\binom{d+1}{r-k}\binom{d+1-r+k}{m+1}-\binom{r+1}{r-k}\binom{d-r}{m-k}\right] \lambda^{m}_{k}(n)^\#.$$
We start by using the induction hypothesis to write the formula out with simplex numbers:
\begin{align*}
\lambda_r^d(n)-\lambda_r^d(n-1)&=\sum\limits^{r}_{k=0}\sum\limits^{d-r+k}_{m=1} \left( \left[\binom{d+1}{r-k}\binom{d+1-r+k}{m+1}-\binom{r+1}{r-k}\binom{d-r}{m-k}\right]\right.\\
&\hspace{1 in} \left.\cdot \sum\limits^{k}_{i=0}(-1)^{k-i}\binom{m+1}{k-i} \alpha^m((i+1)n+k-2i)^\# \right) \\
&\hspace{-.5 in}=\sum\limits^{r}_{k=0}\sum\limits^{d-r+k}_{m=1} \binom{d+1}{r-k}\binom{d+1-r+k}{m+1} \sum\limits^{k}_{i=0}(-1)^{k-i}\binom{m+1}{k-i}\alpha^m((i+1)n+k-2i)^\#\\
&\hspace{-.2 in}-\sum\limits^{r}_{k=0}\sum\limits^{d-r+k}_{m=1} \binom{r+1}{r-k}\binom{d-r}{m-k} \sum\limits^{k}_{i=0}(-1)^{k-i}\binom{m+1}{k-i}\alpha^m((i+1k)n+k-2i)^\#.
\end{align*}
Using Lemma \ref{lem:claim1} for the first sum and Lemma \ref{lem:claim2} for the second sum, we arrive at the desired equation:
\begin{align*}
&\lambda^{d}_{r}(n)-\lambda^{d}_{r}(n-1)\\
&=\sum\limits^{r}_{i=0}(-1)^{r-i}\binom{d+1}{r-i}\alpha^d((i+1)n-r)-\sum\limits^{r}_{i=0}(-1)^{r-i}\binom{d+1}{r-i}\alpha^d((i+1)(n-1)-r).
\end{align*}
This finishes the induction on $n$. Before we finish the induction on $d$, we need to look at the interior sequence.  Here there is no induction on $n$. 
\begin{align*}
\lambda^d_r(n)^\#&=\lambda^d_r(n)-\sum\limits^{r-1}_{k=0}\sum\limits^{d-r+k}_{m=k+1}\binom{d+1}{r-k}\binom{d+1-r+k}{m+1}\lambda^m_k(n)^\#\\
&\hspace{.2 in}-\sum\limits^{d-1}_{m=r+1}\binom{d+1}{0}\binom{d+1}{m+1}\lambda^m_r(n)^\#-\binom{d+1}{r+1}\\
\textrm{(apply Lemma }&\textrm{\ref{lem:claim3} and recall that $\lambda^m_k(n)^\#=0$ if $0<m<k$)}\\
&=\lambda^d_r(n)-\sum\limits^{r-1}_{k=0}\sum\limits^{d-r+k}_{m=1}\binom{d+1}{r-k}\binom{d+1-r+k}{m+1}\lambda^m_k(n)^\#\\
&\hspace{1 in}-\sum\limits^{d-1}_{m=1}\binom{d+1}{0}\binom{d+1}{m+1}\lambda^m_r(n)^\# .
\end{align*}
As we switch to write the expression above with interior simplex numbers, we will add and subtract $\sum_{i=0}^r (-1)^{k-i} \binom{d+1}{r-i} \alpha^d((i+1)n+r-2i)^\#$.
\begin{align*}
\lambda^d_r(n)^\#&=\lambda^d_r(n)+ \sum_{i=0}^r (-1)^{k-i} \binom{d+1}{r-i} \alpha^d((i+1)n+r-2i)^\#  \\
&\hspace{.2 in} -\sum\limits^{r}_{k=0}\sum\limits^{d-r+k}_{m=k+1}\left[ \binom{d+1}{r-k}\binom{d+1-r+k}{m+1}\sum_{i=0}^k \right. \\
&\hspace{1 in}\left. \cdot (-1)^{k-i} \binom{m+1}{k-i} \alpha^m((i+1)n+k-2i)^\#\right] \\
&\textrm{(by the inductive work we can simplify the first sum)}\\
&=\lambda^d_r(n)-\lambda^d_r(n)+\sum_{i=0}^r (-1)^{k-i} \binom{d+1}{r-i} \alpha^d((i+1)n+r-2i)^\#.
\end{align*}
This finishes the inductive step on $d$ and the proof of Theorem \ref{thm:hrect}.
\end{proof}
\section{Proof of Theorem \ref{th:bigalpha}}
\label{sec:thm2}
\begin{proof}
Suppose first that $b\geq a$; we will start with the right-hand side of the formula:
\begin{align*}
\alpha(an-(a-1)-b)&= \sum_{j\geq 0} \sum_{i=0}^{j} (-1)^i \binom{d+1}{i} \binom{d+a(j-i)-b}{a(j-i)-b} \alpha^d(n-j)\\
&= \sum_{j\geq 0} \sum_{i=0}^{j} (-1)^i \binom{d+1}{i} \binom{d+a(j-i-1)-(b-a)}{a(j-i-1)-(b-a)} \alpha^d(n-j)\\
&=\sum_{j\geq -1} \sum_{i=0}^{j+1} (-1)^i \binom{d+1}{i} \binom{d+a(j-i)-(b-a)}{a(j-i)-(b-a)} \alpha^d(n-j-1)\\
&=\sum_{j\geq 0} \sum_{i=0}^{j} (-1)^i \binom{d+1}{i} \binom{d+a(j-i)-(b-a)}{a(j-i)-(b-a)} \alpha^d((n-1-j)\\
&=\alpha(a(n-1)-(a-1)-(b-a)).
\end{align*}
From the second to the third line, we re-index so that $j$ is replaced by $j+1$. We can ignore the terms for $i=j+1$ (including when $j=-1$)  since it will be 0 ($b>0$). 
Using this equality, we can always choose $a$ and $b$ such that $a>b\geq 0$ by replacing $b$ by $b-a$ and $n$ by $n-1$.

We will now use induction on $d$ and $n$. Start with the case that $d=1$ (with $a>b\geq 0$). Begin with the right hand side of the formula:
\begin{align*}
&\; \sum_{j\geq 0} \sum_{i=0}^{j} (-1)^i \binom{2}{i} \binom{1+a(j-i)-b}{a(j-i)-b} \alpha^1(n-j)\\
&= \sum_{j\geq 0} \left[ \binom{1+a(j-1)-b}{a(j-1)-b} -2\binom{1+a(j-1)-b}{a(j-1)-b}+ \binom{1+a(j-2)-b}{a(j-2)-b} \right] (n-j).
\end{align*}
Notice that for the terms where $j\geq 3$, the sum is zero and may be ignored.
Also if $a(j-k)-b\geq 0$, then $\binom{1+a(j-k)-b}{a(j-k)-b}=1+a(j-k)-b$.

If $b>0$, then the term when $j=0$ is zero and may be ignored as well. In this case the sum is 
$$[(1+2a-b)-2(1+a-b)](n-2)+[1+a-b](n-1)=an-a+1-b=\alpha^1 (an-(a-1)-b).$$
On the other hand, if $b=0$, then the $j=2$ term is zero and can be ignored. In this case the sum is
$$[1+a-2(1)](n-1)+[1]n=an-(a-1)=\alpha^1 (an-(a-1)).$$
This finishes the base case for $d=1$.

We now assume the formula holds for $d-1$ and any $n$. To do the inductive step for $d$, we will do induction on $n$. Start by assuming that $n=1$.
The left hand side of the formula is $\alpha^d(a-(a-1)-b)=\alpha^d(1-b)$, which is zero unless $b=0$; if $b=0$, this simplex number is equal to 1.

On the right-hand side of the formula, if $j>0$, the simplex number in the formula is zero. Using this simplification we arrive at the equation below.
\begin{align*}
\sum_{j\geq 0} \sum_{i=0}^{j} &(-1)^i \binom{d+1}{i} \binom{d+a(j-i)-b}{a(j-i)-b} \alpha^d(1-j)\\
&=  \sum_{i=0}^{0} (-1)^i \binom{d+1}{i} \binom{d+a(-i)-b}{a(-i)-b} \alpha^d(1-j)\\
&=  \binom{d+1}{0} \binom{d-b}{-b} \alpha^d(1)=\binom{d-b}{-b}.\\
\end{align*}
Just as with the left hand side of the equation, this formula is zero unless $b=0$ in which case it is equal to 1.

To finish we assume the formula holds for all $\alpha^d(ak-(a-1)-b)$ when $k<n$, and the formula holds for all $n\geq 1$ when the dimension is less than $d$. Now we consider the formula above for $n$ and $d$:
\begin{align*}
\alpha^d(an&-(a-1)-b)=\sum_{k=0}^{a-1}\alpha^{d-1}(an-(a-1)-b-k)+ \alpha^d(a(n-1)-(a-1)-b)\\
&= \sum_{k=0}^{a-1}\sum_{j\geq 0} \sum_{i=0}^d (-1)^i \binom{d}{i} \binom{d-1+a(j-i)-b-k}{a(j-i)-b-k} \alpha^{d-1}(n-j) \\
&\hspace{.2 in} +\sum_{j\geq 0} \sum_{i=0}^{d+1} (-1)^i \binom{d+1}{i} \binom{d+a(j-i)-b}{a(j-i)-b} \alpha^{d}(n-j-1)\\
&= \sum_{j\geq 0} \sum_{i=0}^d (-1)^i \binom{d}{i} \binom{d+a(j-i)-b}{a(j-i)-b}  \alpha^{d-1}(n-j)\\
&\hspace{0.2 in} - \sum_{j\geq 0} \sum_{i=0}^d (-1)^i \binom{d}{i} \binom{d+a(j-i-1)-b}{a(j-i-1)-b} \alpha^{d-1}(n-j)\\
&\hspace{0.2 in} +\sum_{j\geq 0} \sum_{i=0}^{d+1} (-1)^i \binom{d+1}{i} \binom{d+a(j-i)-b}{a(j-i)-b} \alpha^{d}(n-j-1)\\
&\textrm{(re-index the middle term and combine all three terms)}\\
&=\sum_{j\geq 0} \sum_{i=0}^d (-1)^i \binom{d+a(j-i)-b}{a(j-i)-b}\\
&\hspace{.2 in} \left[ \binom{d}{i} \alpha^{d-1}(n-j)+\binom{d}{i-1} \alpha^{d-1}(n-j) + \binom{d+1}{i} \alpha^d (n-j-1)\right]\\
& =\sum_{j\geq 0} \sum_{i=0}^d (-1)^i \binom{d+1}{i}\binom{d+a(j-i)-b}{a(j-i)-b}\alpha^d(n-j).
\end{align*}
This finishes the induction on $d$ and the proof of Theorem \ref{th:bigalpha}.
\end{proof}

\end{document}